\documentclass[a4paper, 12pt, reqno]{article}
\usepackage[cp866]{inputenc}
\usepackage[english]{babel}

\usepackage{amssymb, amsmath, amsthm, enumerate}
\usepackage{verbatim}
\usepackage{indentfirst}

\usepackage{amsfonts,amssymb}
 \newtheorem{thm}{Theorem}[section]
 \newtheorem{cor}[thm]{Corollary}

 \theoremstyle{definition}
 \newtheorem{defn}[thm]{Definition}
 \theoremstyle{remark}
 \newtheorem{rem}[thm]{\bf Remark}
  \newtheorem{assump}[thm]{\bf Assumption}
 \newtheorem{notat}[thm]{\bf Notation}
  \newtheorem{property}[thm]{\bf Property}
    
 \numberwithin{equation}{section}

\begin{document}

\begin{center}{\sc\large The Area Formula for Lipschitz Mappings of
Carnot--Carath\'{e}odory Spaces}
\end{center}

\begin{center}
{\bf Maria Karmanova}
\end{center}

\begin{abstract}
We prove the sub-Riemannian analog of the area formula for Lipschitz (in sub-Riemannian sense)
mappings of equiregular Carnot--Carath\'{e}odory spaces.
\end{abstract}

\tableofcontents

\section{Introduction}

In Euclidean analysis, the well-known area formula {\cite{F}}
$$
\int\limits_{U}\mathcal J(\varphi,x)\,d\mathcal
H^n(x)=\int\limits_{\mathbb
R^k}\sum\limits_{x\in\varphi^{-1}(y)}\chi_U(x)\,d\mathcal H^n(y)
$$
is
proved for large classes of mappings $\varphi:U\to\mathbb R^k$, $U\subset\mathbb R^n$, $n\geq k$,
possessing some regularity
properties. Such classes include continuously differentiable
mappings, Lipschitz mappings; Sobolev mappings and approximately
differentiable mappings with Luzin property $\mathcal N$ (see e.~g.~{\cite{H}}), etc. This formula
is generalized to wide classes of mappings of Riemannian manifolds
and metric spaces~{\cite{F, kirch, AK, km1, km3, km4}}. In many
proofs of the area formula, the approximation of the initial
mapping $\varphi$ by the tangent one
$$w\mapsto D\varphi(x)[w]$$
is essentially used. In particular, the bi-Lipschitz equivalence of the
metrics in the manifold and in the tangent space is applicable in such
proofs.

On equiregular Carnot--Carath\'{e}odory spaces (or simply Carnot manifolds)
(see Definition~{\ref{defcarnot}} below), there are two
structures, namely, Riemannian and {\it sub-Riemannian}. Moreover,
metrics corresponding to these two structures, are not bi-Lipschitz equivalent.
Thus, mappings Lipschitz with respect to sub-Riemannian
metrics may not be Lipschitz with respect to Riemannian metrics,
consequently they may not be differentiable in the classical sense on a set
of non-zero exterior measure. For the mappings of Carnot
manifolds, there exists the specific notion of the {\it
sub-Riemannian differentiability}, or
$hc$-differentiability~{\cite{V2}} of mappings of Carnot manifolds.
\begin{defn}
A mapping $\varphi:U\to\widetilde{\mathbb M}$, $U\subset\mathbb M$, where $\mathbb M$ and $\widetilde{\mathbb M}$ are Carnot manifolds, is $hc$-differentiable
at a point $u\in U$ if there exists a horizontal homomorphism
$L_u:(\mathcal G^u\mathbb M, d_{cc}^u)\to ({\mathcal
G}^{\varphi(u)}\widetilde{\mathbb M}, \widetilde{d}_{cc}^{\varphi(u)})$ of
local Carnot groups, such that
$$
\widetilde{d}_{cc}(\varphi(w),L_u[w])=o(d_{cc}(u,w))\ \text{as $U\cap
{\mathcal G^u\mathbb M}\ni w\to u$}.
$$
Hereinafter, we denote $L_u$ by $\widehat{D}\varphi(u)$.
\end{defn}
In particular cases when both Carnot manifolds are just Carnot groups,
the notion of the $hc$-differential coincides with the one of the
$\mathcal P$-differential introduced by P.~Pansu {\cite{pan1}}.
This definition generalizes the classical definition of differentiability since a local Carnot group approximates the initial Carnot manifold with respect to a sub-Riemannian metric (just like a tangent space approximates a Riemannian manifold with respect to Riemannian metric).
One of results of the
paper~{\cite{V2}} is the following:

\begin{thm}
{\it Lipschitz (in the
sub-Riemannian sense) mappings of Carnot manifolds are $hc$-differentiable almost
everywhere}.
\end{thm}

\begin{thm}[{\cite{V2}}]
Let $\varphi:\mathbb M\to\widetilde{\mathbb M}$ be a contact $C^{1}$-mapping of Carnot manifolds $($in the Riemannian sense$)$. Then, it is continuously $hc$-differentiable everywhere on $\mathbb M$ $($i.~e., its $hc$-differential $\widehat{D}\varphi(u)$ is continuous on $u\in\mathbb M$$)$.
\end{thm}

Nevertheless, up to now, the problem on the area
formula for Lipschitz mappings of Carnot manifolds has been solved only
for some particular cases, i.~e., for mappings of Carnot groups (a
particular case of a Carnot manifold)~{\cite{pauls, mag}} and for
classes of $C^1$-smooth (in the classical sense) contact mappings
of Carnot manifolds~{\cite{vk_area}}. In~{\cite{pauls}}, the
author uses the approximation (with respect to sub-Riemannian
metric) of the initial mapping by the ``tangent'' one defined via
$\mathcal P$-differential.
The main result of {\cite{pauls}} is the following
\begin{thm}[see {\cite[Definition 2.20 and Theorem 3.3]{pauls}}.]\label{area_pauls}
Suppose that $\varphi:\mathbb G\to\widetilde{\mathbb G}$ is a Lipschitz (with respect to sub-Riemannian metrics) map of two Carnot groups. Then, for any $\mathcal H^{\nu}$-measurable set $E\subset\mathbb G$ (here $\nu$ is Hausdorff dimension of $\mathbb G$), we have
$$
\int\limits_E\mathcal J(x)\,d\mathcal H^{\nu}(x)=\int\limits_{\widetilde{\mathbb G}}\sum\limits_{x:\,x\in\varphi^{-1}(y)}\chi_E(y)\,d\mathcal H^{\nu}(y),
$$
where Jacobian $\mathcal J(x)$ equals
\begin{equation}\label{jac_pauls}
\mathcal J(x)=\lim\limits_{t\to0}\Bigl\{\frac{\mathcal H^{\nu}(\varphi(B_{cc}(y,t)))}{\mathcal H^{\nu}(B_{cc}(y,t))}\Bigl|y\in B_{cc}(x,t)\Bigr\},
\end{equation}
where Hausdorff measures are constructed with respect to $d_{cc}$.
\end{thm}
In {\cite{mag}}, the main idea is to use the local
$\mathcal H^{\nu}$-measure distortion under $\varphi$ as a Jacobian:
\begin{thm}[see {\cite[Definition 10 and Theorem 4.4]{mag}}.]\label{area_mag}
Suppose that $\varphi:\mathbb G\to\widetilde{\mathbb G}$ is a Lipschitz (with respect to sub-Riemannian metrics) map of two Carnot groups. Then, for any $\mathcal H^{\nu}$-measurable set $E\subset\mathbb G$ (here $\nu$ is Hausdorff dimension of $\mathbb G$), we have
$$
\int\limits_E\mathcal J_{\nu}(\widehat{D}\varphi(x))\,d\mathcal H^{\nu}(x)=\int\limits_{\widetilde{\mathbb G}}\sum\limits_{x:\,x\in\varphi^{-1}(y)}\chi_E(y)\,d\mathcal H^{\nu}(y),
$$
where Jacobian $\mathcal J_{\nu}(\widehat{D}\varphi(x))$ equals
\begin{equation}\label{jac_mag}
\mathcal J_{\nu}(\widehat{D}\varphi(x))=\frac{\mathcal H^{\nu}(\widehat{D}\varphi(y)[B_{cc}(0,1)])}{\mathcal H^{\nu}(B_{cc}(0,1))},
\end{equation}
where Hausdorff measures are constructed with respect to $d_{cc}$.
\end{thm}
Finally, in~{\cite{vk_area}}, the
sub-Riemannian area formula is derived via the Riemannian one. The result is
\begin{thm}
[The Area Formula for Smooth Mappings
{\cite{vk_area}}]\label{area_smooth} {\it Let $\varphi:\mathbb
M\to\widetilde{\mathbb M}$ be a contact $C^1$-mapping. Then the area formula
$$
\int\limits_{\mathbb
M}f(x)\mathcal J^{SR}(\varphi,x)\,d\mathcal
H^{\nu}(x)=\int\limits_{\widetilde{\mathbb
M}}\sum\limits_{x:\,x\in\varphi^{-1}(y)\cap U}f(x)\,d\mathcal
H^{\nu}(y),
$$
where $f:\mathbb M\to\mathbb E$ $($here $\mathbb E$ is an
arbitrary Banach space$)$ is such that the function
$f(x)\mathcal J^{SR}(\varphi,x)$ is
integrable, and
\begin{equation}\label{jac_sr}
\mathcal J^{SR}(\varphi,x)=\sqrt{\det(\widehat{D}\varphi(x)^*\widehat{D}\varphi(x))}
\end{equation}
is the sub-Riemannian Jacobian of $\varphi$ at $x$, is valid. Here the Hausdorff measures are constructed
with respect to metrics $d_2$ and $\widetilde{d}_2$ with the
multiple~$\omega_{\nu}$.}
\end{thm}

Note that the definition {\eqref{jac_sr}} of the sub-Riemannian Jacobian (that is, its analytic expression via the values of the $hc$-differential) is new even for mappings of Carnot groups.

On the one hand, in view of non-differentiability of Lipschitz in the sub-Riemannian sense mappings, it is impossible to derive the
sub-Riemannian area formula for {\it
arbitrary} Lipschitz mappings of Carnot manifolds via the Riemannian one. On the other
hand, the $hc$-differential $\widehat{D}\varphi$ of a
mapping~$\varphi:\mathbb M\to\widetilde{M}$ at arbitrary point $u$ acts on local Carnot groups $\mathcal G^u\mathbb M$ and $\mathcal G^{\varphi(u)}\widetilde{\mathbb M}$, in which the
sub-Riemannian metrics are not equivalent to the ones in a Carnot
manifold~{\cite{Var}}, thus the relation $d_{\mathbb
M}(\varphi(w),\varphi(v))=(1+o(1))\widehat{d}_{\mathcal
G^{\varphi(u)}}(\widehat{D}\varphi(u)[v],
\widehat{D}\varphi(u)[w])$, where $o(1)\to0$ as $v,w\to u$, near
the point~$u$ cannot be obtained.

In this paper, we give a new
approach to investigation of Lipschitz mappings of Carnot
manifolds based on its $hc$-differentiability only, and its ``partial''
approximation by a ``tangent'' mapping. Such approach is new even for mappings of Euclidean spaces. We prove the area formula
for Lipschitz mappings of Carnot manifolds (see also Theorem~{\ref{arealip}} below):
\begin{thm}\label{area_main}
 {\it Suppose that $D\subset\mathbb M$ is a
measurable set, and the mapping $\varphi:D\to\widetilde{\mathbb M}$ is
Lipschitz with respect to sub-Riemannian quasimetrics $d_2$ and
$\widetilde{d}_2$. Then the area formula
\begin{equation}\label{mainarea}
\int\limits_{D}f(x)\mathcal J^{SR}(\varphi,x)\,d\mathcal
H^{\nu}(x)=\int\limits_{\varphi(D)}\sum\limits_{x:\,x\in\varphi^{-1}(y)}f(x)\,d\mathcal
H^{\nu}(y),
\end{equation}
where $f:D\to\mathbb E$ $($here $\mathbb E$ is an
arbitrary Banach space$)$ is such that the function
$f(x)\mathcal J^{SR}(\varphi,x)$ is
integrable, and the sub-Riemannian Jacobian is the same as in {\eqref{jac_sr}}, is valid. Here the Hausdorff measures are constructed
with respect to metrics $d_2$ and $\widetilde{d}_2$ with the
multiple~$\omega_{\nu}$.}
\end{thm}

\begin{rem}
Note that (see, e.~g., {\cite{pauls}}) that the definitions {\eqref{jac_pauls}} and {\eqref{jac_mag}} are equivalent. Next, it is easy to see that Theorems~{\ref{area_pauls}} and~{\ref{area_mag}} are particular cases of Theorem~{\ref{area_main}}. Indeed, in view of Ball--Box Theorem {\cite{k_dan}}, {\cite{k_appr}}, Hausdorff measures constructed with respect to Carnot--Carathe\'{e}odory metric $d_{cc}$ (see Definition~{\ref{defdcc}}) and with respect to the quasimetric $d_2$ (see Definition~{\ref{defd2}}), are absolutely continuous one with respect to another. Since on a Carnot group these measures are left-invariant, then the derivative of one with respect to another is constant. Denote it by $\mathcal D_{2,cc}$ in the preimage and by $\widetilde{\mathcal D}_{2,cc}$ in the image. In view of the validity of {\eqref{mainarea}} for the mapping $\psi(y)=\widehat{D}\varphi(x)[y]:B_{cc}(0,1)\to\widetilde{\mathbb G}$ we infer
$$
\mathcal J^{SR}(\varphi,x)=\frac{\mathcal H^{\nu}(\widehat{D}\varphi(y)[B_{cc}(0,1)])}{\mathcal H^{\nu}(B_{cc}(0,1))},
$$ where Hausdorff measures are constructed with respect to $d_{2}$. Thus, on the one hand,
\begin{multline*}
\int\limits_{D}\mathcal J^{SR}(\varphi,x)\,d\mathcal
H^{\nu}(x)=\int\limits_{D}\frac{\widehat{\mathcal D}_{2,cc}\mathcal H^{\nu}_{cc}(\widehat{D}\varphi(y)[B_{cc}(0,1)])}{\mathcal D_{2,cc}\mathcal H^{\nu}_{cc}(B_{cc}(0,1))}\,d\mathcal
H^{\nu}(x)\\=\int\limits_{D}\frac{\widehat{\mathcal D}_{2,cc}\mathcal H^{\nu}_{cc}(\widehat{D}\varphi(y)[B_{cc}(0,1)])}{\mathcal D_{2,cc}\mathcal H^{\nu}_{cc}(B_{cc}(0,1))}\mathcal D_{2,cc}\,d\mathcal
H^{\nu}_{cc}(x)\\=\int\limits_{D}{\mathcal D}_{2,cc}J_{\nu}(\widehat{D}\varphi(x))\,d\mathcal
H^{\nu}_{cc}(x)={\mathcal D}_{2,cc}\int\limits_{D}J_{\nu}(\widehat{D}\varphi(x))\,d\mathcal
H^{\nu}_{cc}(x),
\end{multline*}
where Hausdorff measures $\mathcal H^{\nu}_{cc}$ are constructed with respect to $d_{cc}$'s. On the other hand,
\begin{multline*}
\int\limits_{\varphi(D)}\sum\limits_{x:\,x\in\varphi^{-1}(y)}\chi_D(y)\,d\mathcal
H^{\nu}(y)=\int\limits_{\varphi(D)}\sum\limits_{x:\,x\in\varphi^{-1}(y)}\chi_D(y){\mathcal D}_{2,cc}\,d\mathcal
H^{\nu}_{cc}(y)\\={\mathcal D}_{2,cc}\int\limits_{\varphi(D)}\sum\limits_{x:\,x\in\varphi^{-1}(y)}\chi_D(y)\,d\mathcal
H^{\nu}_{cc}(y).
\end{multline*}
Since the value ${\mathcal D}_{2,cc}$ is strictly positive, Theorems~{\ref{area_mag}} and~{\ref{area_pauls}} follow from Theorem~{\ref{area_main}}.
\end{rem}

Emphasize here that although its proof uses in one of its steps the sub-Riemannian area formula for $C^1$-smooth (in the classical sense) mappings, this area formula is not its direct consequence, and its proof requires approaches and methods that are essentially new in comparison with the classical situation of obtaining a result for Lipschitz mappings via the same results for $C^1$-mappings.

Theorem~{\ref{area_smooth}} and the classical area formula for mappings of Riemannian manifolds are particular cases of Theorem~{\ref{area_main}}. Moreover, Theorems~{\ref{area_pauls}} and~{\ref{area_mag}} can also be considered as consequences of Theorem~{\ref{area_main}}. The difference is in definition of Jacobians: in {\cite{pauls}}, {\cite{mag}}, the definition of Jacobian uses measure of balls in Carnot--Carathy\'{e}odory metrics and of their images under $\varphi$ and $\widehat{D}\varphi$. The problem is that the measures of these images cannot be calculated since the structure of Carnot--Carath\'{e}odory balls is unknown in general case. In Theorems~{\ref{area_smooth}} and {\ref{area_main}}, a sub-Riemannian quasimetric is used. Its advantage is that the structure of balls in this quasimetric is well-understandable, and they are easy to work with during the investigation on Jacobian and image properties. Moreover, it allows to write the exact analytic expression of the Jacobian. It is very important for application such as studying extremal surfaces on non-holonomic structures and many others.

\section{Preliminaries}

In this section, we introduce necessary definitions and mention important facts that we will need to prove the main result.

\begin{defn}[compare with
{\cite{G,nsw}}]\label{defcarnot}
Fix a connected Riemannian
$C^{\infty}$-mani\-fold~$\mathbb M$ of a topological dimension~$N$. The
manifold~$\mathbb M$ is called a {\it Carnot--Carath\'{e}o\-dory space} if, in the
tangent bundle $T\mathbb M$, there exists a filtration
$$
H\mathbb M=H_1\mathbb M\subsetneq\ldots\subsetneq H_i\mathbb M\subsetneq\ldots\subsetneq H_M\mathbb M=T\mathbb M
$$
of subbundles of the tangent bundle $T\mathbb M$ such that, for each point $p\in\mathbb M$, there exists a neighborhood $U\subset\mathbb M$
with a collection of $C^1$-smooth vector fields
$X_1,\dots,X_N$ on it enjoying the following two properties. For each
$v\in U$ we have

$(1)$ $H_i\mathbb M(v)=H_i(v)=\operatorname{span}\{X_1(v),\dots,X_{\dim H_i}(v)\}$
is a subspace of $T_v\mathbb M$ of a constant dimension $\dim H_i$,
$i=1,\ldots,M$;

$(2)$
\begin{equation}\label{tcomm}[X_i,X_j](v)=\sum\limits_{\operatorname{deg}
X_k\leq \operatorname{deg} X_i+\operatorname{deg}
X_j}c_{ijk}(v)X_k(v)
\end{equation}
where the {\it degree} $\deg X_k$ equals $\min\{m\mid X_k\in H_m\}$;

If, additionally, the third condition holds then the Carnot--Carath\'eodory space will be called the {\it Carnot manifold}:

$(3)$ a quotient mapping $[\,\cdot ,\cdot\, ]_0:H_1\times
H_j/H_{j-1}\mapsto H_{j+1}/H_{j}$ induced by Lie brackets is an
epimorphism for all $1\leq j<M$.

The subbundle $H\mathbb M$ is called {\it horizontal}.

The number $M$ is called the {\it depth} of the manifold $\mathbb
M$.
\end{defn}

\begin{defn}
Consider Cauchy problem
$$
\begin{cases}
\dot{\gamma}(t)=\sum\limits_{i=1}^Ny_iX_i(\gamma(t)),\ t\in[0,1]\\
\gamma(0)=x,
\end{cases}
$$
where the vector fields $X_1,\ldots,X_N$ are $C^1$-smooth.
Then, for the point $y=\gamma(1)$ we write $y=\exp\Bigl(\sum\limits_{i=1}^Ny_iX_i\Bigr)(x)$.

The mapping $(y_1,\ldots,y_N)\mapsto\exp\Bigl(\sum\limits_{i=1}^Ny_iX_i\Bigr)(x)$ is called {\it exponential}.
\end{defn}

\begin{defn}\label{cs1kind}
Suppose that $u\in\mathbb M$ and $(v_1,\ldots, v_N)\in B_E(0, r)$,
where $B_E(0,r)$ is a Euclidean ball in $\mathbb R^N$. Define a
mapping $\theta_u(v_1,\ldots, v_N):B_E(0,r)\to\mathbb M$ as
follows:
$$
\theta_u(v_1,\ldots,
v_N)=\exp\biggl(\sum\limits_{i=1}^Nv_iX_i\biggr)(u).
$$
It is known, that $\theta_u$ is a $C^1$-diffeomorphism if $0<r\leq
r_u$ for some $r_u>0$. The collection $\{v_i\}_{i=1}^N$ is called
{\it the normal coordinates} or {\it the coordinates of the
$1^{\text{st}}$ kind  $($with respect to $u\in\mathbb M)$} of the
point $v=\theta_u(v_1,\ldots, v_N)$.
\end{defn}


\begin{assump}
Hereinafter, we consider points from a
compactly embedded neighborhood $\mathcal U\Subset\mathbb M$ such
that $\theta_u(B_E(0,r_u))\supset\mathcal U$ for all $u\in\mathcal
U$.
\end{assump}

\begin{defn}
An absolutely continuous curve $\gamma:[0,1]\to\mathbb M$ is called {\it horizontal} if $\dot{\gamma}(t)\in H_{\gamma(t)}\mathbb M$ for almost all $t\in[0,1]$.
\end{defn}

\begin{defn}\label{defdcc}
Carnot--Carath\'{e}odory distance $d_{cc}$ between $x,y\in\mathbb M$ equals
$$
d_{cc}(x,y)=\inf\limits_{\gamma}l(\gamma),
$$
where $\gamma$ is a horizontal curve with endpoints $x$ and $y$.
\end{defn}

Now, we introduce the sub-Riemannian quasimetric locally equivalent to
$d_{cc}$ {\cite{vk_birk}} which simplifies
computations in the main theorems.

\begin{defn}\label{defd2} Let $\mathbb M$ be a Carnot manifold of the
topological dimension $N$ and of the depth $M$, and suppose that
$x=\exp\Bigl(\sum\limits_{i=1}^{N}x_i X_i\Bigr)(u)$. The
quasidistance $d_2(x,g)$ is defined as follows:
\begin{multline*}
d_2(x,u)=\max\Bigl\{\Bigl(\sum\limits_{j=1}^{\dim
H_1}|x_j|^2\Bigr)^{\frac{1}{2}}, \Bigl(\sum\limits_{j=\dim
H_1+1}^{\dim H_2}|x_j|^2\Bigr)^{\frac{1}{2\cdot\operatorname{deg}
X_{\dim H_2}}},\\
\ldots,\Bigl(\sum\limits_{j=\dim
H_{M-1}+1}^{N}|x_j|^2\Bigr)^{\frac{1}{2\cdot\operatorname{deg}
X_{N}}}\Bigr\}.
\end{multline*}
\end{defn}

\begin{rem}
The preimage of a ball
$\mathrm{Box}_2(u,r)=\{x\in\mathbb M_1:\, d_2(x,u)<r\}$ in the
quasimetric $d_2$ under the mapping $\theta_u$ equals
$\mathrm{Box}_2(0,r)=B_E^{n_1}(0,r)\times
B_E^{n_2}(0,r^2)\times\ldots\times B_E^{n_{M}}(0,r^{M})$, where
$B_E^{n_i}$, $i=1,\ldots,M$, are Euclidean balls of the dimensions
$n_i=\dim H_i-\dim H_{i-1}$.
\end{rem}

Such quasimetric is much more easier to deal with than the well known
$d_{\infty}$, where
$d_{\infty}(x,u)=\max\limits_{i=1,\ldots, N}
\{|x_i|^{\frac{1}{\operatorname{deg}X_i}}\}$.
The point is that in the case of $d_{\infty}$, the asymptotical shape of the section of a ball in $d_{\infty}$ by a plane cannot be defined easily since any cube has several sections of different shapes. Since any section of a (Euclidean) ball is just a ball of lower dimension, it is convenient to consider sections of their Cartesian product, i.~e., a ball in $d_2$.

\begin{property} It is easy to see that {\cite{Mit, vk_geom,
vk_birk}} the Hausdorff dimension of $\mathbb M$ with respect to
$d_2$ is equal to $\sum\limits_{i=1}^Mi(\dim H_i-\dim H_{i-1})$,
where $\dim H_0=0$.
\end{property}

\begin{thm}[\cite{vk_birk}]
\label{Liealg} Fix $u\in\mathbb M$. The coefficients
$$
\bar{c}_{ijk}=
\begin{cases} c_{ijk}(u)\text{ of \eqref{tcomm} }& \text{if }\operatorname{deg}
X_i+\operatorname{deg}X_j=\operatorname{deg}X_k \\
0, &\text{otherwise}
\end{cases}
$$
define a graded nilpotent Lie algebra.
\end{thm}

We construct the Lie algebra $\mathfrak{g}^u$ from Theorem
{\ref{Liealg}}  as a graded nilpotent Lie algebra of vector fields
$\{(\widehat{X}_i^u)^{\prime}\}_{i=1}^N$ on $\mathbb R^N$ such that
the exponential mapping $(x_1,\ldots,
x_N)\mapsto\exp\Bigl(\sum\limits_{i=1}^Nx_i(\widehat{X}_i^u)^{\prime}\Bigr)(0)$
equals identity~{\cite{post, blu}}. In view of results of {\cite{fs}}, the value of $(\widehat{X}_j^u)^{\prime}(0)$ is equal to a standard vector $e_{i_j}$, where $i_j\neq i_k$ if $j\neq k$, $j=1,\ldots, N$. We associate to each vector field from the obtained collection such a number $i$ that $(\theta_u)_*\langle(\widehat{X}_i^u)^{\prime}\rangle(u)=X_i(u)$. By the construction, the
relation
\begin{equation}\label{tcommnilp}
[(\widehat{X}_i^u)^{\prime},(\widehat{X}_j^u)^{\prime}]=\sum\limits_{\operatorname{deg}
X_k=\operatorname{deg} X_i+\operatorname{deg}
X_j}c_{ijk}(u)(\widehat{X}_k^u)^{\prime}
\end{equation}
holds for the vector fields
$\{(\widehat{X}_i^u)^{\prime}\}_{i=1}^N$ everywhere on $\mathbb
R^N$.

\begin{notat}
We use the following standard notations: for each $N$-dimen\-sional
multi-index $\mu=(\mu_1,\ldots,\mu_N)$, its {\it homogeneous norm}
equals $|\mu|_h=\sum\limits_{i=1}^N\mu_i\operatorname{deg} X_i$, and
$x^{\mu}=\prod\limits_{i=1}^Nx_i^{\mu_i}$ if $x=(x_1,\ldots, x_N)$
\end{notat}

\begin{defn}
\label{groupoperator} The graded nilpotent Carnot group  ${\mathbb G}_u\mathbb M$
corresponding to the Lie algebra $\mathfrak{g}^u$, is called the
{\it nilpotent tangent cone} of  $\mathbb M$ at $u\in\mathbb M$.
We construct ${\mathbb G}_u\mathbb M$ in $\mathbb R^N$ as a groupalgebra
{\cite{post}},
that is, the exponential map is identical:
$$
\exp\Bigl(\sum\limits_{i=1}^Nx_i(\widehat{X}_i^u)^{\prime}\Bigr)(0)=(x_1,\ldots,x_N).
$$
 By Campbell--Hausdorff formula, the group operation
is defined for the  basis vector
fields~$(\widehat{X}_i^u)^{\prime}$ on $\mathbb R^N$,
$i=1,\ldots,N$, to be left-invariant~{\cite{post}}: if
$$x=\exp\Bigl(\sum\limits_{i=1}^Nx_i(\widehat{X}_i^u)^{\prime}\Bigr),\ y=\exp\Bigl(\sum\limits_{i=1}^Ny_i(\widehat{X}_i^u)^{\prime}\Bigr)
$$
then
$$
x\cdot
y=z=\exp\Bigl(\sum\limits_{i=1}^Nz_i(\widehat{X}_i^u)^{\prime}\Bigr),
$$
where
\begin{equation}\label{groupoperation}
z_i=x_i+y_i+\sum\limits_{\substack{|\mu+\beta|_h=\operatorname{deg}X_i,\notag\\
\mu,\,\beta>0}}F^i_{\mu,\beta}(u)x^{\mu}y^{\beta}.
\end{equation}
\end{defn}

\begin{property}
It is easy to see that
$\widehat{X}_i^u(u)=X_i(u)$, $i=1,\ldots,N$.
\end{property}

\begin{defn} For $u,g\in\mathbb M$, define the exponential mapping
$$\widehat{\theta}^u_g(x_1,\ldots, x_N):B_E(0,r)\to\mathbb M$$ as
$\widehat{\theta}^u_g(x_1,\ldots,
x_N)=\exp\Bigl(\sum\limits_{i=1}^Nx_i\widehat{X}^u_i\Bigr)(g)$, which is a
$C^1$-diffeomorphism for all $0<r\leq r_{u,g}$ for some $r_{u,g}>0$.
\end{defn}

\begin{assump}
We suppose that the neighborhood under consideration $\mathcal U$ is such that $\mathcal U\subset\widehat{\theta}^u_g(B_E(0,r_{u,g}))$ for all $u,g\in\mathcal U$.
\end{assump}

\begin{notat}
The quasimetric $d^u_2$ with respect to the
vector fields $\{\widehat{X}_i^u\}$ is defined similarly to the
initial $d_2$ (defined with respect to $\{{X}_i\}$). A ball in
$d_2^u$ centered at $x$ of a radius $r>0$ is denoted by
$\operatorname{Box}_2^u(x,r)$.
\end{notat}

\begin{notat}
We let the topological dimension of the manifold
$\mathbb M$ ($\widetilde{\mathbb M}$) be equal $N$
($\widetilde{N}$), and we let the Hausdorff dimension with
respect to $d_2$ ($\widetilde{d_2}$) be equal $\nu$ ($\widetilde{\nu}$). The tangent
spaces represented as the direct sums of quotient vector spaces
\begin{multline*}
T_v\mathbb M=\bigoplus\limits_{j=1}^{M}(H_j(v)/H_{j-1}(v)),\
H_0=\{0\},\\
\text{ and }T_u\widetilde{\mathbb
M}=\bigoplus\limits_{j=1}^{\widetilde{M}}(\widetilde{H}_j(u)/\widetilde{H}_{j-1}(u)),\
\widetilde{H}_0=\{0\},
\end{multline*}
at points $v\in\mathbb M$ and $u\in\widetilde{\mathbb M}$, where
$H_1\subset T\mathbb M$ and $\widetilde{H}_1\subset
T\widetilde{\mathbb M}$ are corresponding {\it horizontal}
subbundles, have structures of nilpotent graded Lie algebras
{\cite{vk_geom}}. Denote the dimensions of
${H}_j/{H}_{j-1}~(\widetilde{H}_j/\widetilde{H}_{j-1})$ by symbols
$n_{j}~(\tilde{n}_j)$, $j=1,\ldots,M~(\widetilde{M})$.
\end{notat}

\begin{notat}
Hereinafter, we denote the quasimetric $d_2$ in
the preimage by the symbol $d_2$, and we denote the quasimetric
$d_2$ in the image by the symbol  $\widetilde{d}_2$.
\end{notat}

\begin{assump}\label{assump2}
We suppose that

$1)$ a mapping $\varphi$ is defined on a measurable set $D\subset\mathbb M$;

$2)$  $\dim H_1\leq\dim\widetilde{H}_1$;

$3)$ the
basis vector fields in the preimage and in the image are $C^{1,\alpha}$-smooth, $\alpha>0$, and $\varphi$ is Lipschitz with respect to $d_2$
and $\widetilde{d}_2$
$($$\widetilde{d}_2(\varphi(u),\varphi(v))\leq L d_2(u,v)$ for all
$u,v\in D$ and some $L<\infty$$)$.
\end{assump}

\section{The Main Result}

\begin{thm} [{\cite{V2}}] {\it Suppose that $D\subset\mathbb M$ is a measurable set, and let $\varphi:\mathbb
M\to\widetilde{\mathbb M}$ be a Lipschitz with respect to
sub-Riemannian metrics mapping. Then, it is $hc$-differentiable
almost everywhere. Namely, there exists a horizontal homomorphism
$L_u:(\mathcal G^u\mathbb M, d_{2}^u)\to ({\mathcal
G}^{\varphi(u)}\widetilde{\mathbb M}, d_{2}^{\varphi(u)})$ of
local Carnot groups, such that
$$
\widetilde{d}_{2}(\varphi(w),L_u[w])=o(d_{2}(u,w))\ \text{as $D\cap
{\mathcal G^u\mathbb M}\ni w\to u$}.
$$}
\end{thm}

\begin{defn}
The horizontal homomorphism
$L_u:(\mathcal G^u\mathbb M, d_{2}^u)\to ({\mathcal
G}^{\varphi(u)}\widetilde{\mathbb M}, d_{2}^{\varphi(u)})$ is called the {\it $hc$-differential} of $\varphi$ at $u$.
\end{defn}

\begin{cor}[{\cite{V2}}]\label{diff_cor}
Let $\varphi:\mathbb M\to\widetilde{\mathbb M}$ be a contact $($i.~e., $D\varphi\langle H\rangle\subset\widetilde{H}$$)$ $C^{1}$-mapping of Carnot manifolds $($in the Riemannian sense$)$. Then, it is continuously $hc$-differentiable everywhere on $\mathbb M$.
\end{cor}


\begin{rem}
Using the exponential mapping $\theta_u$, we can consider $L_u$ both as a homomorphism of local Carnot groups, and as a homomorphism of Lie algebras of these local Carnot groups.
\end{rem}

\begin{thm}[Local Approximation Theorem {\cite{vk_birk, k_dan, k_appr}}]\label{lat} Suppose that
$u,w,v\in\mathcal U$, and $d_2(u,w)=O(\varepsilon)$ and
$d_2(u,v)=O(\varepsilon)$. Then we have
$$
|d_2(w,v)-d_2^u(w,v)|=O(\varepsilon^{1+\frac{1}{M}}),
$$
where $O(1)$ is uniform on $\mathcal U$.
\end{thm}

Remark that although the quasimetric in Theorem~{\ref{lat}} is different from the one in~{\cite{vk_birk}}, {\cite{k_dan}} and {\cite{k_appr}} the statement is the same since the scheme of the proof is the same.

\begin{notat}
Denote the $hc$-differential of $\varphi$ at $u$
by the symbol $\widehat{D}\varphi(u)$. Put $Z=\{u\in\mathbb
M:\mathrm{rank}(\widehat{D}\varphi(u))<N\}$.
\end{notat}

\begin{rem}
Given at least one point $u\in\mathbb M$
possessing the property
$\operatorname{rank}\widehat{D}\varphi(u)=N$, the item~2 of
Assumption~{\ref{assump2}} implies
$$\dim H_i-\dim H_{i-1}\leq\dim
\widetilde{H}_i-\widetilde{H}_{i-1},\ i=1,\ldots, M,
$$
where $\dim H_0=0$ and $\dim\widetilde{H}_0=0$. Indeed, it is
enough to take into account the property
$$\widehat{D}\varphi(u)[X,Y]=[\widehat{D}\varphi(u) X,
\widehat{D}\varphi(u) Y],
$$ where $X,Y$ are vector fields
corresponding to the local Carnot group $\mathcal G^u\mathbb M$,
the properties of the local Carnot group~{\cite{vk_geom}}, and property 4 from Definition~{\ref{defcarnot}}.
\end{rem}

\begin{defn}
The (spherical) Hausdorff $\mathcal
H^{\nu}$-measure of a set $E\subset\varphi(\mathbb M)$ is defined
as
$$
{\cal H}^{\nu}(E)=\omega_{\nu}\lim\limits_{\delta\to0}\inf
\Bigl\{\sum\limits_{i\in\mathbb
N}r_i^{\nu}:\bigcup\limits_{i\in\mathbb
N}{\mathrm{Box}_2}(x_i,r_i)\supset E,\,x_i\in
E,\,r_i\leq\delta\Bigr\}.
$$
\end{defn}

\begin{defn}[{\cite{vk_area}}] The {\it sub-Riemannian
Jacobian} equals $${\cal
J}^{SR}(\varphi,x)=\sqrt{\det(\widehat{D}\varphi(x)^*\widehat{D}\varphi(x))}.
$$
\end{defn}

\begin{thm}\label{zeroset} {\it We have $\mathcal
H^{\nu}(\varphi(Z))=0$, where
$$Z=\{x\in\mathbb
M:\,\operatorname{rank}\widehat{D}\varphi(x)<N\}.
$$}
\end{thm}

\begin{proof}  The proof is based on a sharp modification of the arguments
given in~{\cite{V1}}.

Note that
$\widetilde{d}_2^{\varphi(y)}(\varphi(z),\widehat{D}\varphi(y)[z])=o(d_2(y,z))$.
By another words, if  $y\in Z$, then the image of
$\operatorname{Box}_2(y,t)=\operatorname{Box}_2^y(y,t)$ is a
subset of $o(t)$-neighborhood (with respect to
$\widetilde{d}^{\varphi(y)}_2$) of the image of $\mathcal G^y\mathbb
M\cap\operatorname{Box}_2^y(y,t)$ under $\widehat{D}\varphi(y)$.

Since at $y$ we have $\operatorname{rank}\widehat{D}\varphi(y)<N$,
then, the Hausdorff dimension (with respect to
$\widetilde{d}_2^{\varphi(y)}$) of $\widehat{D}\varphi(y)[\mathcal
G^y\mathbb M\cap\operatorname{Box}_2^y(y,t)]$ does not
exceed~$\nu-1$.
Indeed, taking into account the property~3 in Definition~{\ref{defcarnot}}, we have for any basis vector field $X_m\in H_j$, $j>1$, that there exist vector fields $X\in H_{j-1}$ and $Y\in H_1$ such that $X_m=[X,Y]$. Next, $[X,Y](v)=\sum\limits_{l:X_{k_l}\in H_{j-1}, Y_{k_l}\in H_1}a_l(v)[X_{k_l}, Y_{k_l}](v)+Z$, where $Z\in H_{j-1}$. By the second property, $\sum\limits_{l:X_{k_l}\in H_{j-1}, Y_{q_l}\in H_1}a_l(v)[X_{k_l}, Y_{q_l}](v)=\sum\limits_{l:X_{k_l}\in H_{j-1}, Y_{q_l}\in H_1}\sum\limits_{p:\operatorname{deg}X_p\leq j}a_l(v)c_{k_lq_lp}(v)X_p(v)$. By the property of vector fields of the local Carnot group,
$$\sum\limits_{l:X_{k_l}\in H_{j-1}, Y_{q_l}\in H_1}a_l(v)[\widehat{X}^y_{k_l}, \widehat{Y}^y_{q_l}](v)=\sum\limits_{l:X_{k_l}\in H_{j-1}, Y_{q_l}\in H_1}\sum\limits_{p:\operatorname{deg}X_p=j}a_l(v)c_{k_lq_lp}(y)\widehat{X}_p(v),
$$ and, by assumption $X_m=[X,Y]$ we have that the last sum equals $\widehat{X}^u_m$.
In view of the property
$$\widehat{D}\varphi(y)[X,Y]=[\widehat{D}\varphi(y) X,
\widehat{D}\varphi(y) Y],
$$ where $X,Y$ are vector fields
corresponding to the local Carnot group $\mathcal G^y\mathbb M$,
we infer that the sum of degrees of the images under
$\widehat{D}\varphi(y)$ of the basis vector fields cannot be
bigger than $\nu$; moreover, it equals $\nu$ only if
$\operatorname{rank}\widehat{D}\varphi(y)=N$. In all the other
cases, this sum does not exceed $\nu-1$.

For $0<\sigma<\infty$, take $\varepsilon>0$, and suppose without
loss of generality that  $Z$ is compact, and that both values
$o(1)$ in the definition of $hc$-differentiability and in Local
Approximation Theorem {\ref{lat}} do not exceed $\varepsilon$. Fix
$\delta>0$ and construct the covering of $Z$ by balls
$\{\operatorname{Box}_2(y_i,t_i)\}_{i\in\mathbb N,\, y_i\in Z}$,
$t_i\leq\delta$, from the definition of $\mathcal
H^{\nu}_{\delta}$, such that
$$
\omega_{\nu}\sum\limits_{i\in\mathbb N}t_i^{\nu}\leq\mathcal
H^{\nu}(Z)+\sigma.
$$

Fix $i\in\mathbb N$ and estimate $\mathcal
H^{\nu}(\varphi(\operatorname{Box}_2(y_i,t_i)))$. The image
$\varphi(\operatorname{Box}_2(y_i,t_i))$ is a subset of a
$\varepsilon t_i$-neighborhood of
$\widehat{D}\varphi(y_i)[\operatorname{Box}_2^{y_i}(y_i,t_i)]$
which has sub-Riemannian Hausdorff dimension $\nu_i$ not exceeding
$\nu-1$. Consider the family of balls
$$\{\operatorname{Box}^{\varphi(y_i)}_2(s,2\varepsilon
t_i)\}_{s\in\widehat{D}\varphi(y_i)[\operatorname{Box}_2^{y_i}(y_i,t_i)]}
$$
in $\mathcal G^{\varphi(y_i)}\widetilde{\mathbb M}$ with centers
on the set
$\widehat{D}\varphi(y_i)[\operatorname{Box}_2^{y_i}(y_i,t_i)]$ and
of radii $2\varepsilon t_i$, which covers the set
$\varphi(\operatorname{Box}_2(y_i,t_i))$. In view of the
degeneracy of $\widehat{D}\varphi(y_i)$,  we have that the volume
of the intersection
$$\operatorname{Box}^{\varphi(y_i)}_2(s,2\varepsilon t_i)\cap
\widehat{D}\varphi(y_i)[\operatorname{Box}_2^{y_i}(y_i,t_i)]
$$ is
not less than $O((\varepsilon t_i)^{\nu_i})$, where $O(1)$ is
strictly greater than zero uniformly on some compact neighborhood
(here we also take into account the left-invariance on $\mathcal
G^{\varphi(y_i)}\widetilde{\mathbb M}$, and we suppose without
loss of generality that $Z$ is a subset of such compact
neighborhood). By  Lipschitzity of $\varphi$ and degeneracy of
$\widehat{D}\varphi(y_i)$,  we have that the volume of $
\widehat{D}\varphi(y_i)[\operatorname{Box}_2^{y_i}(y_i,t_i)]$ does
not exceed~$O(t_i^{\nu_i})$ (here $O(1)$ is also uniform on the
compact neighborhood under consideration). Here $\nu_i\leq\nu-1$
depends on the degeneracy of $\widehat{D}\varphi$ at $y_i$,
namely, it equals sum of degrees of all the basis vector fields in
$\mathcal G^u\mathbb M$ on which $\widehat{D}\varphi(y_i)$ is
non-degenerate, and images of which are independent.

Since in the local
Carnot group $\mathcal G^{\varphi(y_i)}\widetilde{\mathbb M}$ the
quasimetric $\widetilde{d}^{\varphi(y_i)}_2$ is locally equivalent
to Carnot--Carath\'{e}odory metric
$\widetilde{d}^{\varphi(y_i)}_{cc}$, then we obtain applying
5$r$-Covering Lemma that there exist not more than
$$\frac{O(t_i^{\nu_i})}{O((\varepsilon t_i)^{\nu_i})}=\frac{1}{O(\varepsilon^{\nu-1})}$$
of balls
$\{\operatorname{Box}_2^{\varphi(y_i)}(s_j,r_j)\}_{j\in\mathbb N}$
covering $\varphi(\operatorname{Box}_2(y_i,t_i))$, the radii of
which vary from $2\varepsilon t_i$ to $l\cdot2\varepsilon t_i$,
and such that the corresponding balls
$\{\operatorname{Box}_2^{\varphi(y_i)}(s_j,2\varepsilon
t_i)\}_{j\in\mathbb N}$  are disjoint. Here the constant $l$
depends on the equivalence coefficients of
$\widetilde{d}^{\varphi(y_i)}_2$ and
$\widetilde{d}^{\varphi(y_i)}_{cc}$, and of 5$r$-Covering Lemma
{\cite{F}}, and $O(1)$ is strictly greater than zero uniformly in
$i\in\mathbb N$.

In view of Local Approximation Theorem {\ref{lat}} for
$\widetilde{d}_2$ and $\widetilde{d}_2^{\varphi(y_i)}$ (we may
assume without loss of generality that on the set
$\varphi(\operatorname{Box}_2(y_i,t_i))$ we have
$|\widetilde{d}_2-\widetilde{d}_2^{\varphi(y)}|\leq\varepsilon
t_i$),  the collection of the balls
$\{\operatorname{Box}_2(s_j,2r_j)\}_{j\in\mathbb N}$ covers the
set $\varphi(\operatorname{Box}_2(y_i,t_i))$. Consequently,
$$
\mathcal
H^{\nu}_{4l\varepsilon\delta}(\varphi(\operatorname{Box}_2(y_i,t_i)))\leq
(4l\varepsilon
t_i)^{\nu}\cdot\frac{1}{O(\varepsilon^{\nu-1})}=O(\varepsilon)\cdot
t_i^{\nu},
$$
where $O(\cdot)$ is uniform on $\varphi(Z)$. Thus,
$$
\mathcal H^{\nu}_{4l\varepsilon\delta}(\varphi(Z))\leq\mathcal
H^{\nu}_{4l\varepsilon\delta}\Bigl(\bigcup\limits_{i\in\mathbb
N}\varphi(\operatorname{Box}_2(y_i,t_i))\Bigr)\leq
O(\varepsilon)\sum\limits_{i\in\mathbb N}t_i^{\nu}\leq
O(\varepsilon)(\mathcal H^{\nu}(Z)+\sigma).
$$
Here $O(1)$ is uniform in all $\delta>0$ small enough.
If $\delta\to0$ then we have $\varepsilon\to0$, and the theorem
follows.
\end{proof}

\begin{thm}[The Area Formula for Smooth Mappings
{\cite{vk_area}}]\label{th_area} {\it Let $\varphi:\mathbb
M\to\widetilde{\mathbb M}$ be a contact $C^1$-mapping which is
continuously $hc$-differentiable $($i.~e., its $hc$-differential
$\widehat{D}\varphi(u)$ is continuous on $u\in\mathbb M$$)$
everywhere. Then the area formula
$$
\int\limits_{\mathbb
M}f(x)\sqrt{\det(\widehat{D}\varphi(x)^*\widehat{D}\varphi(x))}\,d\mathcal
H^{\nu}(x)=\int\limits_{\widetilde{\mathbb
M}}\sum\limits_{x:\,x\in\varphi^{-1}(y)}f(x)\,d\mathcal
H^{\nu}(y),
$$
where $f:\mathbb M\to\mathbb E$ $($here $\mathbb E$ is an
arbitrary Banach space$)$ is such that the function
$f(x)\sqrt{\det(\widehat{D}\varphi(x)^*\widehat{D}\varphi(x))}$ is
integrable, is valid. Here the Hausdorff measures are constructed
with respect to metrics $d_2$ and $\widetilde{d}_2$ with the
multiple~$\omega_{\nu}$.}
\end{thm}

\begin{thm}\label{arealip}
 {\it Suppose that $D\subset\mathbb M$ is a
measurable set, and the mapping $\varphi:D\to\widetilde{\mathbb M}$ is
Lipschitz with respect to sub-Riemannian quasimetrics $d_2$ and
$\widetilde{d}_2$. Then the area formula
$$
\int\limits_{D}f(x)\sqrt{\det\bigl(\widehat{D}\varphi^*(x)\widehat{D}\varphi(x)\bigr)}\,d\mathcal
H^{\nu}(x)=\int\limits_{\varphi(D)}\sum\limits_{x:\,x\in\varphi^{-1}(y)}f(x)\,d\mathcal
H^{\nu}(y),
$$
where $f:D\to\mathbb E$ $($here $\mathbb E$ is an
arbitrary Banach space$)$ is such that the function
$f(x)\sqrt{\det(\widehat{D}\varphi(x)^*\widehat{D}\varphi(x))}$ is
integrable, is valid. Here the Hausdorff measures are constructed
with respect to metrics $d_2$ and $\widetilde{d}_2$ with the
multiple~$\omega_{\nu}$.}
\end{thm}

\begin{proof} {\sc $1^{\text{\sc st}}$ Step.} Without loss of generality, we may assume that $D\subset\mathcal U$. In view of Theorem
{\ref{zeroset}} we have $\mathcal H^{\nu}(\varphi(Z))=0$. It is
left to prove the area formula for the set $A=D\setminus Z$. We
may assume without loss of generality~{\cite{F}} that on the
measurable set $A$ we have
\begin{equation}\label{bilip}
C_1d_2(u,v)\leq\widetilde{d}_2(\varphi(u),\varphi(v))\leq
C_2d_2(u,v)
\end{equation}
for some $0<C_1,C_2<\infty$,
$\operatorname{rank}\widehat{D}\varphi(z)=N$ for all points of
$hc$-differentiability of the mapping $\varphi$, and the set $A$
has the finite measure. For convenience, consider the case
$f\equiv1$. Note that the set function defined on open sets in $E\subset\mathbb M$,
$$
\Phi(E)=\int\limits_{\varphi(E\cap A)}\,d\mathcal H^{\nu}(y)
$$
is absolutely continuous (since $\varphi$ is a Lipschitz mapping:
indeed, it is easy to see that there exists such
$Q=Q(\varphi)<\infty$ that
$$\mathcal H^{\nu}(\varphi(E\cap
D))\leq Q\mathcal H^{\nu}(E\cap D)$$ for any set $E$) and
additive. Consequently {\cite{F}},
$$
\Phi(A)=\int\limits_{A}\Phi^{\prime}(x)\,d\mathcal H^{\nu}(x).
$$
Our goal is to show that
$$
\Phi^{\prime}(y)=\sqrt{\det(\widehat{D}\varphi(y)^*\widehat{D}\varphi(y))}
$$
almost everywhere.

{\sc $2^{\text{\sc nd}}$ Step.} For each $\varepsilon>0$, there
exists a set $\Sigma_{\varepsilon}$ of the $\mathcal
H^{\nu}$-measure not exceeding $\varepsilon$, such that on
$A\setminus\Sigma_{\varepsilon}$ the mapping $\varphi$ is
continuously $hc$-differentiable, i.~e., the $hc$-differential
$\widehat{D}\varphi(z)$, $z\in A\setminus\Sigma_{\varepsilon}$, is
continuous {\cite[Lemma~4.6]{V2}}. The definition of the
$hc$-differentiability implies for $w,u\in A$ (here by our assumption $u$ is a point of $hc$-differentiability of $\varphi$):
$$
\widetilde{d}_2(\varphi(w),\varphi(u))=\widetilde{d}_2(\widehat{D}\varphi(u)[w],\varphi(u))+o(d_2(w,u)),
$$
where $o(1)\to0$ as $w\to u$. We also may assume without loss of generality that $o(\cdot)$ is uniform in $u\in A\setminus\Sigma_{\varepsilon}$.
Since we have the assumption that
$d_2(w,u)\geq\frac{1}{C_2}\widetilde{d}_2(\varphi(w),\varphi(u))$,
it follows that
\begin{equation}\label{eq1}
\widetilde{d}_2(\varphi(w),\varphi(u))(1+o(1))=\widetilde{d}_2(\widehat{D}\varphi(u)[w],\varphi(u)).
\end{equation}
Here $o(1)$ is uniform in $u\in A\setminus\Sigma_{\varepsilon}$.

{\sc $3^{\text{\sc rd}}$ Step.} Fix $\varepsilon>0$ and prove the
area formula for $A_{\varepsilon}=A\setminus\Sigma_{\varepsilon}$.
Hereinafter in this proof, for the set $E\subset\mathbb M$, $E\varsubsetneq
A_{\varepsilon}$, the symbol $\varphi(E)$ denotes $\varphi(E\cap
A_{\varepsilon})$. Fix $\sigma>0$ and $r>0$, and consider the set
$\Delta_{\sigma r^{\nu}}$, $\mathcal H^{\nu}(\Delta_{\sigma
r^{\nu}})<\sigma r^{\nu}$, such that on
$A_{\varepsilon}\setminus\Delta_{\sigma r^{\nu}}$, measurable
functions
$$\Psi_m(y)=\frac{m^{\nu}}{\omega_{\nu}}
\int\limits_{\operatorname{Box}_2(y,1/m)\cap
A_{\varepsilon}}\,d\mathcal H^{\nu}(x),
$$ $m\in\mathbb N$,
converge uniformly to the unity~{\cite{F}}.

We may assume without
loss of generality that $\Delta_{\sigma
r^{\nu}}\subset\Delta_{\sigma t^{\nu}}$ for $r<t$. Indeed, it
is sufficient to construct for each $l\in\mathbb N$ a set
$\widetilde{\Delta}_l$, $\mathcal H^{\nu}$-measure of which does
not exceed
$\sigma\Bigl(\frac{1}{(l)^{\nu}}-\frac{1}{(l+1)^{\nu}}\Bigr)$,
and such that functions $\Psi_m$ converge uniformly to the unity on $A_{\varepsilon}\setminus\widetilde{\Delta}_l$. Next, for $r\in(1/l,1/(l-1)]$, put
$\Delta_{\sigma
r^{\nu}}=\bigcup\limits_{k=l}^{\infty}\widetilde{\Delta}_k$; it
is easy to see that its $\mathcal H^{\nu}$-measure is not more
than $\sigma/l^{\nu}<\sigma r^{\nu}$. Moreover, for $r<t$
we have $\Delta_{\sigma r^{\nu}}\subset\Delta_{\sigma
t^{\nu}}$. We will need this property at the end of the proof when $r,\sigma\to0$ not to ``loose'' points we have considered.

Take $r>0$ small enough, a density point $g\in
A_{\varepsilon}\setminus\Delta_{\sigma r^{\nu}}$ of the set
$A_{\varepsilon}$, and an open ball $\operatorname{Box}_2(g,r)$.
Since $\mathcal H^{\nu}(\Delta_{\sigma r^{\nu}})\leq\sigma
r^{\nu}$, it follows that $\mathcal H^{\nu}(\varphi(\Delta_{\sigma
r^{\nu}}))\leq Q\sigma r^{\nu}$, and
$$\mathcal H^{\nu}(\varphi(\operatorname{Box}_2(g,r)))\leq
\mathcal
H^{\nu}(\varphi(\operatorname{Box}_2(g,r)\setminus\Delta_{\sigma
r^{\nu}}))+Q\sigma r^{\nu}.
$$

Fix $\tau>0$ and (for fixed $\sigma$ and $r$) choose such
$\delta\leq\delta_0(\tau,\sigma,r)$, $\delta_0\in(0,\sigma r)$, that
for $1/m\leq\min\{\delta,\delta C_1\}$, where $C_1$ is taken from {\eqref{bilip}}, we have
$\Psi_m(y)\geq 1-\tau$ for $y\in\operatorname{Box}_2(g,r)\cap
A_{\varepsilon}\setminus\Delta_{\sigma r^{\nu}}$ (it is possible in view of the uniform convergence of $\Psi_m(y)$ to the unity on $A_{\varepsilon}\setminus\Delta_{\sigma r^{\nu}}$).

For the chosen
$\delta>0$, construct the covering
$\{\operatorname{Box}_2(x_i,r_i)\}_{i\in\mathbb N}$ of the image
$\varphi(\operatorname{Box}_2(g,r)\setminus\Delta_{\sigma
r^{\nu}})$ from the definition of $\mathcal H^{\nu}_{\delta}$.

The definition of the set $\Delta_{\sigma r^{\nu}}$ implies that
for any covering by balls
$\{\operatorname{Box}_2(x_i,r_i)\}_{i\in\mathbb N}$ of the image
$\varphi(\operatorname{Box}_2(g,r)\setminus\Delta_{\sigma
r^{\nu}})$ from the definition of $\mathcal H^{\nu}$-measure,
the centers $x_i$ are images of the points
$y_i\in\operatorname{Box}_2(g,r)\cap
A_{\varepsilon}\setminus\Delta_{\sigma r^{\nu}}$ that are
density points of the set $\operatorname{Box}_2(g,r)\cap
A_{\varepsilon}$.

{\sc $4^{\text{\sc th}}$ Step.} To each point
$y_i=\varphi^{-1}(x_i)$, assign the $\mathcal P$-differentiable
mapping $\eta_i$ of local Carnot groups defined as follows:
$\mathcal G^{y_i}\mathbb M\ni
w\overset{\eta_i}\mapsto\widehat{D}\varphi(y_i)[w]\in\mathcal
G^{\varphi(y_i)}\widetilde{\mathbb M}$.
Each such mapping belongs to the class $C^1$ (in the classical sense),
and it is contact (as a mapping of local Carnot groups) since
$\eta_i(w)=\theta_{x_i}\circ L\circ\theta_{y_i}^{-1}(w)$. Here
the linear mapping $L$ is defined by the matrix of the
$hc$-differential $\widehat{D}\varphi(y_i)$ in the following sense: first, the mapping $\theta_{y_i}^{-1}$ ``calculates'' the coordinates of $w$ with respect to $y_i$, then the linear mappings $L$ matrix of which coincides with the matrix of the $hc$-differential $\widehat{D}\varphi$ in the bases $\{\widehat{X}^{y_i}_j\}_{j=1}^N$ and  $\{\widehat{X}^{x_i}_k\}_{k=1}^{\widetilde{N}}$ acts on the obtained point of $\mathbb R^N$, and finally, the mapping $\theta_{x_i}$ assigns a point on $\mathbb M$ to this image point from $\mathbb R^N$. Recall that in view of the property $$\widehat{D}\varphi(y_i)[X,Y]=[\widehat{D}\varphi(y_i) X,
\widehat{D}\varphi(y_i) Y],
$$ where $X,Y$ are vector fields
corresponding to the local Carnot group $\mathcal G^{y_i}\mathbb M$ in the definition of the $hc$-differential $\widehat{D}\varphi$,
its matrix has block-diagonal structure in the bases $\{\widehat{X}^{y_i}_j\}_{j=1}^N$ and  $\{\widehat{X}^{x_i}_k\}_{k=1}^{\widetilde{N}}$. Besides of this,
$\eta_i$ is continuously $\mathcal P$-differentiable (see Corollary~{\ref{diff_cor}}; {\cite{pan1}}), and
$\widehat{D}\eta_i(v)=\widehat{D}\varphi(y_i)$ for all $v$ close
enough to~$y_i$.
Indeed,
$d_2^{x_i}(\eta_i(w),\widehat{D}\varphi(y_i)[w])=0=o(d_2^{y_i}(v,w))$.

Next, in the definition of the value
$$\mathcal H^{\nu}_{\delta}(\varphi(\operatorname{Box}_2(g,r)\setminus\Delta_{\sigma
r^{\nu}})),
$$ to
each ball $\operatorname{Box}_2(x_i,r_i)$, there corresponds the
summand $\omega_{\nu}r_i^{\nu}$. Fix $i\in\mathbb N$. In view of
the area formula for Carnot groups {\cite{vk_area}} (see
Theorem~{\ref{th_area}}), for each mapping $\eta_i$, $i\in\mathbb
N$, we have
$$\omega_{\nu}r_i^{\nu}=\sqrt{\det(\widehat{D}\varphi^*(y_i)\widehat{D}\varphi(y_i))}
\cdot\widehat{{\mathcal
H}^{\nu}}^{y_i}(\eta_i^{-1}(\operatorname{Box}_2(x_i,r_i))),
$$ where the
symbol $\widehat{{\mathcal
H}^{\nu}}^{y_i}$ denotes $\mathcal
H^{\nu}$-measure in the local Carnot group~$\mathcal
G^{y_i}\mathbb M$ with respect to $d_2^{y_i}$.

Now, consider the sets
$\eta_i^{-1}(\operatorname{Box}_2(x_i,r_i))$ and
$\varphi^{-1}(\operatorname{Box}_2(x_i,r_i))\cap A_{\varepsilon}$.
Note that, under the mapping $\eta_i$, the preimage of an open set
$\operatorname{Box}_2(x_i,r_i)$ is also open, moreover, it has a
boundary consisting of the finite number of surfaces of the
class~$C^1$. In view of {\eqref{eq1}} (for $u=y_i$), all points of
the set $\varphi^{-1}(\operatorname{Box}_2(x_i,r_i))\cap
A_{\varepsilon}$ are contained in an $o(r_i)$-neighborhood of the
set $\eta_i^{-1}(\operatorname{Box}_2(x_i,r_i))$.

Indeed, it follows from the fact that if $w\in \varphi^{-1}(\operatorname{Box}_2(x_i,r_i))\cap
A_{\varepsilon}$ then
\begin{equation}\label{phieta}
d_2(\varphi(w),x_i)=(1+o(1))d_2(\eta_i(w),x_i),
\end{equation}
and consequently $\eta_i(w)\in\operatorname{Box}_2(x_i, r_i(1+o(1)))$. Here $o(1)$ is uniform in all $i\in\mathbb N$ due to the choice of $A_{\varepsilon}$.

Besides of this,
according to {\eqref{eq1}} (for $u=y_i$), all the points of the
set $A_{\varepsilon}$, lying inside
$\eta_i^{-1}(\operatorname{Box}_2(x_i,r_i))$ and such that the
distance to $\partial[\eta_i^{-1}(\operatorname{Box}_2(x_i,r_i))]$
is more than $o(r_i)$, belong to the set
$\varphi^{-1}(\operatorname{Box}_2(x_i,r_i))\cap A_{\varepsilon}$. Indeed, if
$$
w\in\eta_i^{-1}(\operatorname{Box}_2(x_i,r_i(1-o(1))))\ \text{
then }\ \widetilde{d}_2(\varphi(w),x_i)\leq r
$$
for suitable values of $o(1)$ (see {\eqref{phieta}}). Here $o(1)$ is uniform in all $i\in\mathbb N$.

Since $y_i\in\operatorname{Box}_2(g,r)\cap
A_{\varepsilon}\setminus\Delta_{\sigma r^{\nu}}$, we have
\begin{multline}\label{eq2}
{\mathcal
H}^{\nu}(\eta_i^{-1}(\operatorname{Box}_2(x_i,r_i)))(1+o(1))\\\geq\mathcal
H^{\nu}(\varphi^{-1}(\operatorname{Box}_2(x_i,r_i))\cap A_{\varepsilon})
\geq\mathcal H^{\nu}(\eta_i^{-1}[\operatorname{Box}_2(x_i,r_i(1-o(1)))]\cap A_{\varepsilon})\\
\geq(1-o(1)){\mathcal
H}^{\nu}(\eta_i^{-1}(\operatorname{Box}_2(x_i,r_i)))-\tau(1+o(1))\mathcal H^{\nu}(\operatorname{Box}_2(x_i,r_i/C_1)),
\end{multline}
where $o(1)\to0$ as $r_i\to0$ uniformly in all $x_i$, $i\in\mathbb N$, by the choice of $\delta>0$.

{\sc $5^{\text{\sc th}}$ Step.} Theorem~{\ref{th_area}}, the equalities $\widehat{D}\eta_i\equiv\widehat{D}\varphi(y_i)$, $i\in\mathbb N$, and the continuity of
the $hc$-differential $\widehat{D}\varphi$ imply
\begin{multline}\label{eq3}
\sum\limits_{i\in\mathbb
N}\omega_{\nu}r_i^{\nu}=\sum\limits_{i\in\mathbb
N}\sqrt{\det(\widehat{D}\varphi^*(y_i)\widehat{D}\varphi(y_i))}
\cdot\widehat{{\mathcal
H}^{\nu}}^{y_i}(\eta_i^{-1}(\operatorname{Box}_2(x_i,r_i)))\\
=
\Bigl(\sqrt{\det(\widehat{D}\varphi^*(g)\widehat{D}\varphi(g))}(1+o(1))\Bigr)\cdot\sum\limits_{i\in\mathbb
N}\widehat{{\mathcal
H}^{\nu}}^{y_i}(\eta_i^{-1}(\operatorname{Box}_2(x_i,r_i))),
\end{multline}
where $o(1)\to0$ as $r\to0$. Thus, the sum
$\sum\limits_{i\in\mathbb N}\omega_{\nu}r_i^{\nu}$ is close to
the minimal value if and only if the sum $\sum\limits_{i\in\mathbb
N}\widehat{{\mathcal
H}^{\nu}}^{y_i}(\eta_i^{-1}(\operatorname{Box}_2(x_i,r_i)))$ is also
close to its minimal value. Since on $\mathbb M$ we have $\mathcal
H^{\nu}=\widehat{{\mathcal
H}^{\nu}}^{y_i}(1+o(1))$ where $o(1)\to0$ as the points of a measured set converge to
$y_i$ (it is enough to consider their expressions via Riemannian measures),  $i\in\mathbb N$, we may consider the sum $\sum\limits_{i\in\mathbb
N}{\mathcal
H}^{\nu}(\eta_i^{-1}(\operatorname{Box}_2(x_i,r_i)))$ instead of $\sum\limits_{i\in\mathbb
N}\widehat{{\mathcal
H}^{\nu}}^{y_i}(\eta_i^{-1}(\operatorname{Box}_2(x_i,r_i)))$. Now, we calculate this value. Since
the ``balls''
\begin{multline*}
\{\eta^{-1}(\operatorname{Box}_2(x,t)):\ x=\varphi(y),\
y\in\operatorname{Box}_2(g,r)\cap
A_{\varepsilon}\setminus\Delta_{\sigma r^{\nu}},\\
\eta(w)=\widehat{D}\varphi(y)[w],\ t\in(0,\min\{\delta, \delta C_1\}),\
\eta^{-1}(\operatorname{Box}_2(x,t))\subset\operatorname{Box}_2(g,r)\}
\end{multline*}
have the doubling condition (with respect to the measure $\mathcal
H^{\nu}$ in view of the relation $\mathcal
H^{\nu}=\widehat{{\mathcal
H}^{\nu}}^{y_i}(1+o(1))$, see above), then Vitali Covering Theorem implies the existence of the collection
$\{\eta_i^{-1}(\operatorname{Box}_2(x_i,r_i))\}_{i\in\mathbb N}$,
covering the set $\operatorname{Box}_2(g,r)\cap
A_{\varepsilon}\setminus\Delta_{\sigma r^{\nu}}$ up to a set of $\mathcal H^{\nu}$-measure zero. For this (remaining) set, there exists an at most countable covering by ``balls''
$\{\eta_j^{-1}(\operatorname{Box}_2(x_j,t_j))\}_{j\in\mathbb N}$,
with the sum of their $\mathcal H^{\nu}$-measures less than $\sigma r^{\nu}$.

Relation {\eqref{eq1}} implies that
$\bigcup\limits_{i\in\mathbb
N}\operatorname{Box}_2(x_i,\tilde{r}_i)\cup\bigcup\limits_{j\in\mathbb
N}\operatorname{Box}_2(x_j,\tilde{t}_j)
\supset\varphi(\operatorname{Box}_2(g,r)\setminus\Delta_{\sigma
r^{\nu}})$, where $\tilde{r}_i=r_i(1+o(1))$ and
$\tilde{t}_j=t_j(1+o(1))$, and $o(1)$ are uniform in all $i,j$.
Moreover, the sum $S$ of $\mathcal H^{\nu}$-measures of the preimages of these balls under the corresponding mappings $\eta_i$ can be estimated as
$$
[\mathcal
H^{\nu}(\operatorname{Box}_2(g,r))+\sigma r^{\nu}](1+o(1))\geq
S \geq\mathcal H^{\nu}(\operatorname{Box}_2(g,r)\cap
A_{\varepsilon}\setminus\Delta_{\sigma r^{\nu}}),
$$ where
$o(1)\to0$ as $r_i,\,t_j\to0$, $i,j\in\mathbb N$. In view of
{\eqref{eq2}}, we have
\begin{multline}\label{sumest}
S\cdot(1+o(1))\geq\sum\limits_{i\in\mathbb
N}\mathcal H^{\nu}(\varphi^{-1}(\operatorname{Box}_2(x_i,r_i))
\cap A_{\varepsilon}) \\
+\sum\limits_{j\in\mathbb N}\mathcal
H^{\nu}(\varphi^{-1}(\operatorname{Box}_2(x_j,t_j)) \cap
A_{\varepsilon})\geq(1-o(1)-O(\tau))S,
\end{multline}
where $o(1)\to0$ as $\delta\to0$, and $O(1)$ is bounded uniformly on $A_{\varepsilon}$.

Note that, the sum $\sum\limits_{k\in\mathbb N}\mathcal
H^{\nu}(\varphi^{-1}(\operatorname{Box}_2(x_k,r_k)) \cap
A_{\varepsilon})$ cannot be less than $\mathcal H^{\nu}(\operatorname{Box}_2(g,r)\cap
A_{\varepsilon}\setminus\Delta_{\sigma r^{\nu}})$ in case of any covering of
$\varphi(\operatorname{Box}_2(g,r)\setminus\Delta_{\sigma
r^{\nu}})$ by any collection
$\{\operatorname{Box}_2(x_k,r_k)\}_{k\in\mathbb N}$.
Consequently, we have for the value (see {\eqref{sumest}})
\begin{multline*}
\frac{\mathcal H^{\nu}(\operatorname{Box}_2(g,r)\cap
A_{\varepsilon}\setminus\Delta_{\sigma r^{\nu}})}{1+o(1)}
\leq S\\\leq\frac{1}{1-o(1)-O(\tau)}[\mathcal
H^{\nu}(\operatorname{Box}_2(g,r)\cap
A_{\varepsilon}\setminus\Delta_{\sigma r^{\nu}})+O(1)\sigma r^{\nu}],
\end{multline*} where the values $O(1)$
are bounded uniformly in all small $\delta>0$ and small $r>0$, and $o(1)\to0$ as $\delta\to0$,
is indeed close to the minimal one.

Since we have $\tau\to0$ and $o(1)\to0$ while $\delta\to0$, and while $r\to0$ we can take
$\sigma=\sigma(r)\to0$, then, taking into account the fact that $g$ is the density point of
the set $A_{\varepsilon}$, and $\mathcal
H^{\nu}(\Delta_{\sigma r^{\nu}})=o(r^{\nu})$, where $o(1)\to0$
as $r\to0$, we deduce from {\eqref{eq3}} that
$$
\lim\limits_{r\to0}\frac{\mathcal
H^{\nu}(\varphi(\operatorname{Box}_2(g,r)\cap
A_{\varepsilon}))}{\mathcal
H^{\nu}(\varphi(\operatorname{Box}_2(g,r)\cap
A_{\varepsilon}\setminus\Delta_{\sigma r^{\nu}}))}=1
\text{ and }\lim\limits_{r\to0}\frac{S}{\omega_{\nu}r^{\nu}}=1.
$$
Consequently, {\eqref{eq2}} implies
$$
\Phi^{\prime}(g)=\lim\limits_{r\to0}\frac{\mathcal
H^{\nu}(\varphi(\operatorname{Box}_2(g,r)\cap
A_{\varepsilon}))}{\omega_{\nu}r^{\nu}}
=\sqrt{\det(\widehat{D}\varphi^*(g)\widehat{D}\varphi(g))}.
$$
Since the latter is valid for almost all $g\in A_{\varepsilon}$, it implies the area formula for the set $A_{\varepsilon}$.
We use standard argument to derive the area formula for the set~$A$. The theorem follows.

\begin{rem}
All results of this paper are also true for mappings of Carnot
manifolds enjoying conditions from {\cite[Remark~2.2.19]{vk_birk}}
with basis vector fields on $\mathbb M$ belonging to
$C^{1,\alpha}$, $\alpha>0$, and basis vector fields on
$\widetilde{\mathbb M}$ belonging to $C^{1,\widetilde{\alpha}}$,
$\widetilde{\alpha}>0$.
\end{rem}

\section*{Acknowledgement}

The author thanks Professor Sergey Vodopyanov for fruitful
discussions that have given rise to new ideas, and for interest to work.

{The research was supported by RFBR (Grant
10-01-00662-a) and the State Maintenance Program for supporting Scientific School and Young Scientists
(Grant NSh-6613.2010.1).}
\end{proof}

\end{document}